\documentclass[10pt]{amsart}

\usepackage{amsfonts, mathrsfs}
\usepackage{amsthm}
\usepackage{amsmath}
\usepackage{graphicx}
\usepackage{amscd,amssymb,amsthm}

\title{Probability distribution built by Prabhakar function. Related Tur\'an and Laguerre inequalities}

\author{Tibor K. Pog\'any}
\address{Faculty of Maritime Studies, University of Rijeka, Rijeka 51000, Croatia \\
Applied Mathematics Institute, \'Obuda University, 1034 Budapest, Hungary}
\email{poganj@pfri.hr} 

\author{\v Zivorad Tomovski}
\address{Department of Mathematics, University of Rijeka, Radmile Matej\v ci\'c 2, 51000 Rijeka, Croatia\\
(On leave) Faculty of Natural Sciences and Mathematics, Institute of Mathematics, Saints Cyril and Methodius 
University, 1000 Skopje, Macedonia}



\newtheorem{lemma}{Lemma}
\newtheorem{theorem}{Theorem}
\newtheorem{corollary}{Corollary}
\newtheorem{proposition}{Proposition}
\newtheorem{remark}{Remark}

\begin{document}

\maketitle
\allowdisplaybreaks

\begin{abstract}
Introducing the discrete probability distribution by means of the Prabhakar (or the three--parameter Mittag--Leffler)  
function, we establish explicit expressions for raw and factorial moments and also general fractional order moments. Applying an   
elementary moment inequality we obtain functional upper bounds for the Tur\'anian difference for Prabhakar function. Finally,    
a Laguerre inequality is proved and functional upper bound has been given for Laguerreian difference for Prabhakar function. 
\end{abstract}

\allowdisplaybreaks

\section{Introduction}

The three--parameter Mittag--Leffler function was introduced by Prabhakar \cite{Prabhakar} in the form 
   \begin{equation} \label{M1}
	    E_{\alpha, \beta}^\gamma(z) = \sum_{k \geq 0} \frac{(\gamma)_k}{\Gamma(\alpha k+\beta)} 
			            \frac{z^k}{k!}, \qquad \min\{ \Re(\alpha),  \Re(\beta), \Re(\gamma)\}>0;\, z \in \mathbb C.
	 \end{equation}
Here and in what follows 
   \[ (\lambda)_{\mu} = \dfrac{\Gamma(\lambda+\mu)}{\Gamma(\lambda)} = 
			    \begin{cases}
             1 &  \mu = 0;\, \lambda \in \mathbb C \setminus \{0\}\\ 
             \lambda(\lambda+1) \cdots (\lambda + n-1) & \mu=n \in \mathbb N;\; \lambda \in \mathbb C
          \end{cases} \,; \]
it being understood conventionally that $(0)_0:= 1$, where the latter case stands for the standard familiar Pochhammer symbol 
(or shifted factorial). For comprehensive treatment of $E_{\alpha, \beta}^\gamma(z)$ consult e.g. \cite{Tomovski}. 

For $\gamma =1$, we recover with \eqref{M1} the two--parameter Mittag--Leffler function $E_{\alpha, \beta }(z)$ defined by 
   \[ E_{\alpha, \beta }(z) = \sum_{k \geq 0} \frac{z^k}{\Gamma(\alpha k + \beta)}, 
			                        \qquad \Re (\alpha )>0, \beta \in \mathbb{C};\,z\in \mathbb C \, . \]
Moreover, the case $\beta = \gamma =1$ yields the classical Mittag--Leffler function 
   \[ E_{\alpha }(z)=\sum_{k\geq 0}\frac{z^{k}}{\Gamma (\alpha k+1)},\qquad \Re(\alpha )>0;\,z\in \mathbb{C}. \]

\section{Fractional Poisson process with Prabhakar's three--parameter Mittag--Leffler function}

The Mittag-Leffler function appeared as residual waiting time between events in renewal processes already in the sixties of the 
past century, namely processes with properly scaled thinning out the sequence of events in a power law renewal process. Such a process 
in essences is a fractional Poisson process \cite{Beghin}, \cite{Laskin1}, \cite{Laskin2}, \cite{Mainardi}, \cite{Politi} 
which in the recent years has become subject of intensive research. The telegraph process is a finite--velocity one dimensional 
random motion of which various probabilistic features are known. 

The fractional extensions of the telegraph process $\{ \mathscr T_\alpha(t) \colon t \geq 0\}$, whose changes of 
direction are related to the fractional Poisson process $\{ \mathscr N_\alpha(t) \colon t \geq 0\}$ having distribution \cite{Beghin}
   \[ \mathbb P(\mathscr N_\alpha(t) = k) = \frac{\lambda^k}{E_\alpha(\lambda t^\alpha)} \frac{t^{\alpha k}}{\Gamma(\alpha k+1)}, 
			                                       \qquad k \in \mathbb{N}_{0}:=\mathbb{N}\cup \{0\};\,t\geq 0. \]
However, we point out that the fractional Poisson process resulting in $\{ \mathscr N_{\alpha, \beta}(t) \colon t\geq 0\}$, defined by 
the two--parameter Mittag--Leffler function $E_{\alpha, \beta}(\lambda t^\alpha)$ was studied by Herrmann \cite{Herrmann}. The 
related distribution reads
   \[ \mathbb P(\mathscr N_{\alpha, \beta}(t) = k) = \frac{\lambda^k}{E_{\alpha, \beta}(\lambda t^\alpha)} 
			                      \frac{t^{\alpha k}}{\Gamma (\alpha k + \beta)},\qquad k\in \mathbb{N}_{0};\,t\geq 0, \]
for which the related raw moments are obtained in terms of the Bell polynomials, see \cite[p. 5]{Herrmann}. 

Here we introduce a more general fractional Poisson process $\{\mathscr N_{\alpha, \beta}^\gamma(t) \colon t\geq 0\}$ defined by the 
Prabhakar (or in another words three--parameter Mittag--Leffler) function $E_{\alpha, \beta}^\gamma(\lambda t^\alpha)$ 
having distribution
   \begin{equation} \label{Y0}
	    \mathbb P(\mathscr N_{\alpha, \beta}^\gamma(t) = k) = \frac{\lambda^k}{E_{\alpha, \beta}^\gamma(\lambda t^\alpha)} 
			           \frac{(\gamma)_k\, t^{\alpha k}}{k!\, \Gamma(\alpha k+\beta)}, \qquad 
								 k \in \mathbb N_0;\, t \geq 0. 
	 \end{equation}
At this point it is worth to mention that there is a correspondence  between the non--homogeneous Poisson process 
$\{ \mathscr N(t) \colon t \geq 0\}$ with intensity function $\lambda \alpha t^{\alpha-1}$, that is 
   \[ \mathbb P( \mathscr N(t) = k) = {\rm e}^{-\lambda t^\alpha} \, \frac{(\lambda t^\alpha)^k}{\Gamma(k+1)}, 
	               \qquad \lambda, \alpha>0; k \in \mathbb N_0\, ,\]
and the Prabhakar--type fractional Poisson process $\{\mathscr N_{\alpha, \beta}^\gamma(t) \colon t\geq 0\}$. This connection is 
exposed in the following result.

\begin{proposition} Let $\min\{ \alpha,  \beta, \gamma, \lambda\}>0$ and $t \geq 0$. Then 
   \[ \mathbb P(\mathscr N_{\alpha, \beta}^\gamma(t) = k) = \dfrac{\dfrac{(\gamma)_k}{\Gamma(\alpha k + \beta)}\,
	              \mathbb P( \mathscr N(t) = k)}{\sum\limits_{n \geq 0} \dfrac{(\gamma)_n}{\Gamma(\alpha n + \beta)}\,
								\mathbb P( \mathscr N(t) = n)}, \qquad k \in \mathbb N_0\,,\]
where $\mathscr N(t)$ is a non--homogeneous Poisson process with intensity function $\lambda \alpha t^{\alpha-1}$.
\end{proposition}

\begin{proof} Rewriting \eqref{Y0} into 
   \[ \mathbb P(\mathscr N_{\alpha, \beta}^\gamma(t) = k) = \frac{\dfrac{(\gamma)_k}{\Gamma(\alpha k + \beta)}\, 
	              \dfrac{(\lambda t^\alpha)^k}{k!}\, {\rm e}^{-\lambda t^\alpha}}{\sum\limits_{n \geq 0} 
								\dfrac{(\gamma)_n}{\Gamma(\alpha n + \beta)}\, \dfrac{(\lambda t^\alpha)^n}{n!}\, {\rm e}^{-\lambda t^\alpha}}\,, \] 
we arrive at the statement. 
\end{proof}

In accordance with \eqref{Y0} we shall study the probability distribution associated with the Prabhakar function 
$E_{\alpha, \beta}^\gamma(\lambda t^\alpha)$ letting  $\lambda = 1$ throughout for the sake of simplicity. So, consider a non--negative 
random variable $X$ on a standard probability space $(\Omega, \mathscr F, \mathbb P)$ having the fractional 
Poisson--type distribution 
    \[ \mathbb P_{\alpha, \beta}^\gamma(k) = \mathbb P(X = k) = \frac1{E_{\alpha, \beta}^\gamma(t^\alpha) } 
			           \frac{(\gamma)_k \,t^{\alpha k}}{k!\, \Gamma(\alpha k+\beta)}, \qquad 
								 k \in \mathbb N_0;\, t \geq 0, \]
where remains $\min\{\alpha, \beta, \gamma\} >0$. Being $\sum_{k \geq 0} \mathbb P_{\alpha, \beta}^\gamma(k) = 1$, the rv $X$ i
s well defined. This correspondence we quote in the sequel $X \sim {\rm ML}(\alpha, \beta, \gamma)$. 

Let us remind that the factorial moment of the rv $X$ of order $s \in \mathbb N$ is given by
   \[ \Phi_s = \mathbb E X(X-1) \cdots (X-s+1) = (-1)^s\, \mathbb E(-X)_s 
	           = \frac{{\rm d}^s}{{\rm d}t^s}\left(\mathbb E\,t^X\right)\Big|_{t=1}\,,\]
provided the moment generating function $M_X(t) = \mathbb E\, t^X$ there exists in some neighborhood of $t =1$ together will all 
its derivatives up to the order $s$. By virtue of the Vi\`ete--Girard formulae for expanding $X(X-1) \cdots (X-s+1)$ we obtain
   \[ \Phi_s = \sum _{r=1}^s (-1)^{s-r} \, e_r\, \mathbb EX^r\,;\]
here $e_r$ represents elementary symmetric polynomials:
   \[ e_r = e_r(\ell_1, \cdots, \ell_r) = \sum_{1 \le \ell_1< \cdots <\ell_r \le s-1}
                  \ell_1 \cdots \ell_r, \qquad r = \overline{0,s-1}.\]

\begin{theorem}
For all $\min\{ \Re(\alpha),  \Re(\beta), \Re(\gamma)\}>0$ the $s$--th raw moment of the rv $X \sim {\rm ML}(\alpha, \beta, \gamma)$ 
reads as follows
   \begin{equation} \label{N3}
	    \mathbb E X^s = \frac1{E_{\alpha, \beta}^\gamma(t^\alpha) } \sum_{j=0}^s (\gamma)_j\,{s \brace j} \,t^{\alpha j}\, 
			                E_{\alpha, \alpha j+\beta}^{\gamma+j}(t^\alpha), \qquad s \in \mathbb N_0,\, t \geq 0.
	 \end{equation}
Moreover, the $s$--th factorial moment 
   \begin{equation} \label{N5}
	    \Phi_s = \frac1{E_{\alpha, \beta}^\gamma(t^\alpha) } \sum _{r=1}^s (-1)^r \, e_r\,
			                   \sum_{j=0}^r (\gamma)_j\,{r \brace j} \,t^{\alpha j}\, E_{\alpha, \alpha j+\beta}^{\gamma+j}(t^\alpha)\,.
	 \end{equation}
\end{theorem}

\begin{proof} Recall the connection between the raw and the factorial moments of a rv:
   \[ \mathbb E X^s = \sum_{j=0}^s (-1)^j\,{s \brace j} \mathbb E(-X)_j, \qquad 
	                    {s \brace j} = \frac1{j!} \sum_{m=0}^j (-1)^{j-m} \binom{j}{m} m^s\,,\]
where the curly braces denote Stirling numbers of the second kind. As to the stated result we need only the direct calculation:
   \begin{align*} 
	    \mathbb E X^s &= \sum_{j=0}^s (-1)^j\, {s \brace j} \mathbb E(-X)_j 
			               = \sum_{j=0}^s (-1)^j\,{s \brace j} \sum_{k \geq 0} (-k)_j \, \mathbb P_{\alpha, \beta}^\gamma(k)  \\
			              &= \frac1{E_{\alpha, \beta}^\gamma(t^\alpha) } \sum_{j=0}^s (-1)^j\,{s \brace j} \sum_{k \geq 0} 
										   \frac{(-k)_j \,(\gamma)_k \,t^{\alpha k}}{k!\, \Gamma(\alpha k+\beta)}  \\
										&= \frac1{E_{\alpha, \beta}^\gamma(t^\alpha) \, \Gamma(\gamma)} \sum_{j=0}^s {s \brace j}\, t^{\alpha j}\,
										   \sum_{k \geq j} \frac{\Gamma(\gamma+(k-j)+j) \,t^{\alpha(k-j)}}{(k-j)!\, 
											 \Gamma(\alpha (k-j)+ \alpha j+\beta )}  \\
										&= \frac1{E_{\alpha, \beta}^\gamma(t^\alpha) } \sum_{j=0}^s \frac{\Gamma(\gamma+j)}{\Gamma(\gamma)} \,
										   \,{s \brace j} \,t^{\alpha j} \,\sum_{k \geq 0} \frac{(\gamma+j)_k \,t^{\alpha k}}
											 {k!\, \Gamma(\alpha k+ \alpha j+\beta)} \,,
	 \end{align*}
which is the statement \eqref{N3}. The derivation of \eqref{N5} is now obvious. 
\end{proof} 

To obtain the fractional order moments we need the so-called extended Hurwitz--Lerch Zeta (HLZ) function introduced by Gupta {\it et al.} 
\cite{GGOS} (see also \cite[p. 491, Eq. (1.20); p. 503, Eq. (6.2)]{Srivastava}) $\Phi^{(\rho, \sigma, \kappa)}_{\lambda, \mu; 
\nu}(z, s, a)$ is given by
   \begin{equation} \label{N6}
      \Phi^{(\rho, \sigma, \kappa)}_{\lambda, \mu; \nu}(z, s, a) = \sum_{n \geq 0} 
             \frac{(\lambda)_{\rho n}\, (\mu)_{\sigma n}}{n!\, (\nu)_{\kappa n}}\, \dfrac{z^n}{(n+a)^s}\,,
   \end{equation}
where $\lambda, \mu \in \mathbb C;\, a, \nu \in \mathbb C \setminus \mathbb Z_0^-$; $\rho, \sigma, \kappa >0$; $\kappa-\rho-\sigma+1>0$ 
when $s, z \in \mathbb C$; $\kappa-\rho-\sigma = -1$ and $s \in \mathbb C$ when $|z|< \delta = \rho^{-\rho} \sigma^{-\sigma} 
\kappa^\kappa$; while $\kappa-\rho-\sigma = -1$ and $\Re(s + \nu - \lambda - \mu) >1$ when $|z|= \delta$. Specially by letting 
$\sigma \to 0$ in \eqref{N6} we arrive at the generalized HLZ function $\Phi^{(\rho, 0, \kappa)}_{\lambda, \mu; \nu}(z, s, a) \equiv 
\Phi^{(\rho, \kappa)}_{\,\,\lambda; \nu}(z, s, a)$ which is of importance one for our next result.

\begin{theorem} Let $X \sim {\rm ML}(\alpha, \beta, \gamma)$. For all $\min\{\alpha, \beta, \gamma\}>0$ and for all $s \geq 0$ 
we have
   \begin{equation} \label{N7}
	    \mathbb E X^s = \frac{\gamma\,t^\alpha}{E_{\alpha, \beta}^\gamma(t^\alpha)\, \Gamma(\alpha+\beta)}\, 
										  \Phi^{(1, \alpha)}_{\gamma+1; \alpha+\beta}(t^\alpha, 1-s, 1)\,.
   \end{equation}
\end{theorem}

\begin{proof} By definition, for all $s>0$ we have
   \[ \mathbb E X^s = \frac1{E_{\alpha, \beta}^\gamma(t^\alpha)} \sum_{n \geq 1} n^s \, 
			                 \frac{(\gamma)_n \,t^{\alpha n}}{n!\, \Gamma(\alpha n+\beta)}\,, \]
since the therm zeroth vanishes. In turn {\it mutatis mutandis}
   \begin{align*}
	    \mathbb E X^s &= \frac1{E_{\alpha, \beta}^\gamma(t^\alpha)} \sum_{n \geq 1}  \, 
			                 \frac{n^{s-1}\, (\gamma)_n\, t^{\alpha n}}{(n-1)!\, \Gamma(\alpha n+\beta)} \\
										&= \frac{\gamma\,t^\alpha}{E_{\alpha, \beta}^\gamma(t^\alpha)} \sum_{n \geq 0}  \, 
			                 \frac{(\gamma+1)_n\, t^{\alpha n}}{n!\, \Gamma(\alpha n+\alpha+\beta)\,(n+1)^{1-s}} \\
										&= \frac{\gamma\,t^\alpha}{E_{\alpha, \beta}^\gamma(t^\alpha)\, \Gamma(\alpha+\beta)}\, 
										   \Phi^{(1, \alpha)}_{\gamma+1; \alpha+\beta}(t^\alpha, 1-s, 1)\,.
	 \end{align*}
Now, being $\lambda = \gamma+1, \nu=\alpha+\beta, z=t^\alpha, s \mapsto 1-s, \rho = 1, \kappa = \alpha$ and $a=1$, by applying the 
convergence constraints for $\Phi^{(\rho, \kappa)}_{\,\,\lambda; \nu}(z, s, a)$ listed around \eqref{N6}, we finish the proof 
of the assertion. 
\end{proof} 

\begin{remark} It is worth to mention that Herrmann \cite[p. 5, Eq. (3.10)]{Herrmann} has been derived the raw integer order moments 
for the two--parameter Mittag--Leffler distributed rv which is covered with our $Y \sim {\rm ML}(\alpha, \beta, 1)$ distribution, 
in the form 
   \[ \mathbb EY^n = \frac1{E_{\alpha, \beta}(t)}\, \left( t\, \frac{\rm d}{{\rm d}t}\right)^n\, E_{\alpha, \beta}(t), 
	                   \qquad n \in \mathbb N_0\, . \]
Recalling that $E_{\alpha, \beta}^1(t) = E_{\alpha, \beta}(t)$, our relations \eqref{N3} and \eqref{N7} cover this result. Indeed, taking 
$s=1, \gamma=1, t \mapsto t^\frac1\alpha$ in \eqref{N3} we have 
   \[ \mathbb EX = t\, \frac{E_{\alpha, \alpha+\beta}^2(t)}{E_{\alpha, \beta}(t)}\,.\] 
On the other hand, since 
   \[ \left(E_{\alpha, \beta}(t)\right)' = \sum_{n \geq 1} \frac{n\, t^{n-1}}{\Gamma(\alpha n+\beta)} 
	               = \sum_{n \geq 1} \frac{(2)_{n-1} t^{n-1}}{(n-1)!\,\Gamma(\alpha (n-1) + \alpha + \beta)} 
								 = E_{\alpha, \alpha+\beta}^2(t)\,,\]
we see that $\mathbb EX \equiv \mathbb EY$. 

Next, by setting $s=1, \gamma=1, t \mapsto t^\frac1\alpha$ formula \eqref{N7} becomes
   \[ \mathbb E X = t\,\frac{\Phi^{(1, \alpha)}_{2; \alpha+\beta}(t, 0, 1) }{E_{\alpha, \beta}(t)\, \Gamma(\alpha+\beta)}										
									= \frac{t}{E_{\alpha, \beta}(t)\, \Gamma(\alpha+\beta)}\, 
									  \sum_{n \geq 0} \frac{(2)_n t^n}{n!\,(\alpha+\beta)_{\alpha n}}.\]
The rest is obvious. \hfill  $\Box$
\end{remark}

\begin{remark}
In fact, Theorem 2 is the fractional (in $s$) counterpart of the first claim \eqref{N3} of Theorem 1. However, the advantage of 
preceding is that it consists from a finite $s+1$--term linear combination of three parameter Mittag--Leffler functions of 
Prabhakar type. Specifying now $s \in \mathbb N_0$ in \eqref{N7}, we obtain a new set of summation formulae for a special class 
of extended HLZ functions $\Phi^{(1, \alpha)}_{\gamma+1; \alpha+\beta}(t^\alpha, 1-s, 1)$ in terms of related Prabhakar function. 
\hfill $\Box$
\end{remark}

\begin{corollary} For all $\min\{\alpha, \beta, \gamma\}>0$ and for all $s \in \mathbb N_0$ we have
   \[ \Phi^{(1, \alpha)}_{\gamma+1; \alpha+\beta}(t^\alpha, 1-s, 1) = \frac{\Gamma(\alpha+\beta)}{\gamma\,t^\alpha}\,
	              \sum_{j=0}^s (\gamma)_j\,{s \brace j} \,t^{\alpha j}\, E_{\alpha, \alpha j+\beta}^{\gamma+j}(t^\alpha)\,, 
								\qquad t >0.\] 
\end{corollary} 

\section{Functional upper bounds for a Tur\'anian difference built by Prabhakar function}

Our next goal is to derive a Tur\'an type inequality for the Prabhakar function. In this purpose let us introduce the 
generalized Tur\'anian difference 
   \[ \Delta_T(t) = \left[E_{\alpha, \alpha+\beta}^{\gamma+1}(t^\alpha)  \right]^2 
						-    E_{\alpha, \beta}^\gamma(t^\alpha)\,E_{\alpha, 2\alpha+\beta}^{\gamma+2}(t^\alpha),\qquad t>0\, ,\]
which was built with respect to all three parameters in the initial Mittag--Leffler function $E_{\alpha, \beta}^\gamma(t^\alpha)$. To 
establish a functional upper bound for $\Delta_T(t)$ we study firstly the first two raw moments, using  Theorem 2. 
The resulting expressions are
   \begin{align} 
	    \mathbb E X   &= \frac{\gamma\,t^\alpha}{E_{\alpha, \beta}^\gamma(t^\alpha)\, \Gamma(\alpha+\beta)}\, 
										   \Phi^{(1, \alpha)}_{\gamma+1; \alpha+\beta}(t^\alpha, 0, 1) \notag \\
			              &= \frac{\gamma\,t^\alpha\, E_{\alpha, \alpha+\beta}^{\gamma+1}(t^\alpha)}
										   {E_{\alpha, \beta}^\gamma(t^\alpha)} \\ \label{N9}
			\mathbb E X^2 &= \frac{\gamma\,t^\alpha}{E_{\alpha, \beta}^\gamma(t^\alpha)\, \Gamma(\alpha+\beta)}\, 
										   \Phi^{(1, \alpha)}_{\gamma+1; \alpha+\beta}(t^\alpha, 1-s, 1) \notag \\
			              &= \frac1{E_{\alpha, \beta}^\gamma(t^\alpha)} \sum_{n \geq 1}  \, 
			                 \frac{[(n-1)+1]\, (\gamma)_n\, t^{\alpha (n-1+1)}}{(n-1)!\, \Gamma(\alpha (n-1)+\alpha+\beta)} \notag \\
										&= \frac{\gamma\,t^\alpha}{E_{\alpha, \beta}^\gamma(t^\alpha)} \left[ E_{\alpha, \alpha+\beta}^{\gamma+1}(t^\alpha) 
										 + (\gamma+1) t^\alpha \,E_{\alpha, 2\alpha+\beta}^{\gamma+2}(t^\alpha)  \right]\, .
	 \end{align}
Making use of the basic moment property $\mathbb EX^2 \geq (\mathbb EX)^2$ valid for all finite second order moment random variables, 
we deduce: 
   \[ E_{\alpha, \beta}^\gamma(t^\alpha)\, \left[ E_{\alpha, \alpha+\beta}^{\gamma+1}(t^\alpha) 
						+    (\gamma+1) t^\alpha \,E_{\alpha, 2\alpha+\beta}^{\gamma+2}(t^\alpha)  \right] 
						\geq \gamma\,t^\alpha\, \left[E_{\alpha, \alpha+\beta}^{\gamma+1}(t^\alpha)  \right]^2\,, \]
which is equivalent to 
   \begin{equation} \label{N9}
	    E_{\alpha, \beta}^\gamma(t^\alpha)\, \left[ E_{\alpha, \alpha+\beta}^{\gamma+1}(t^\alpha) 
						+    t^\alpha \,E_{\alpha, 2\alpha+\beta}^{\gamma+2}(t^\alpha) \right] 
						\geq \gamma\,t^\alpha\, \Delta_T(t) \, .
	 \end{equation}
	
\begin{lemma} For all $\min\{\alpha, \beta, \gamma\}>0$ we have 
   \begin{equation} \label{N10}
		  E_{\alpha, \beta}^\gamma(t^\alpha) \leq \frac1{\Gamma_0}\, \frac1{(1-t^\alpha)^\gamma}\,, \qquad t\in (0,1)\, ,
	 \end{equation}
where $\Gamma_0 = \min\limits_{n \in \mathbb N_0} \Gamma(\alpha n+\beta)$.

Moreover, for all $\alpha \geq 1,\, \beta \geq t_0,\, \gamma>0$, where $(t_0, \Gamma(t_0))$ is the unique positive minimum of 
the $\Gamma$ function, there holds
   \begin{equation} \label{N11}
		  E_{\alpha, \beta}^\gamma(t^\alpha) \leq \frac1{\Gamma(\beta)} \, {}_1F_1(\gamma; \beta; t^\alpha)\,, \qquad t>0\, .
	 \end{equation}
Here ${}_1F_1(a; b; z)$ stands for the Kummer's confluent hypergeometric function 
   \[ {}_1F_1(a; b; z) = \sum_{k \geq 0} \frac{(a)_k}{(b)_k} \frac{z^k}{k!}\, . \]
\end{lemma} 

\begin{proof} As to the proof of \eqref{N10} we only have to remind that in the right half--plane there is the unique 
$\min \Gamma(t_0) \approx 0.885603$ at $t_0 \approx 1.46163$, say. (Obviously we cannot guarantee that $x_0 \in \{ \alpha k+\beta 
\colon k \in \mathbb N_0\}$). So 
   \[ E_{\alpha, \beta}^\gamma(t^\alpha) \leq \frac1{\min\limits_{k \in \mathbb N_0} \Gamma(\alpha k+\beta)} 
	                                            \sum_{k \geq 0} (\gamma)_k \frac{t^{\alpha k}}{k!} 
																				 =    \frac1{\Gamma_0}\, \frac1{(1-t^\alpha)^\gamma},\qquad t\in (0,1)\, . \]
When $\alpha \geq 1$ and $\beta \geq t_0$, then obviously $\Gamma(\alpha k+\beta) \geq \Gamma(\beta),\, k \in \mathbb N_0$. Then 
   \[ E_{\alpha, \beta}^\gamma(t^\alpha) \leq \frac1{\Gamma(\beta)} \,\sum_{k \geq 0} \frac{(\gamma)_k}{(\beta)_k} \frac{z^k}{k!}\,,\]
which coincides with \eqref{N11}.
\end{proof}

Now, combining \eqref{N9}, \eqref{N10} and/or \eqref{N11} we arrive at

\begin{theorem} For all $\min\{\alpha, \beta, \gamma\}>0$ and for all $t \in (0,1)$ we have
   \[ \Delta_T(t) \leq \frac1{\gamma \Gamma_0\, t^\alpha (1-t^\alpha)^{2\gamma+1}}\,
	                     \left[ \frac1{\Gamma_1} + \frac{t^\alpha}{\Gamma_2\, (1-t^\alpha)}\right], \]
where $\Gamma_p = \min\limits_{n \in \mathbb N_p} \Gamma(\alpha n+\beta),\, p=0,1,2$; $\mathbb N_q = \{ q, q+1, q+2, \cdots \}, 
q \in \mathbb N_0$. 

Moreover, for all $\alpha \geq 1,\, \beta \geq t_0,\, \gamma>0$ and for all $t>0$ there holds 
   \begin{align*} 
	    \Delta_T(t) &\leq \frac{{}_1F_1(\gamma; \beta; t^\alpha)}{\gamma\,\Gamma(\beta)\, t^\alpha} \,
	                      \left[ \frac1{\Gamma(\alpha+\beta)}\,{}_1F_1(\gamma+1; \alpha+\beta; t^\alpha) \right. \\
									&\qquad + \left. \frac{t^\alpha }{\Gamma(2\alpha+\beta)}\,{}_1F_1(\gamma+2; 2\alpha+\beta; t^\alpha)\right]\,, 
	 \end{align*}
where $(t_0, \Gamma(t_0))$ is the positive minimum of the $\Gamma$--function.
\end{theorem} 

\section{Toward Laguerre type inequality results} 

The Laguerreian difference (see e.g. the classical paper \cite{Skovgaard} or the recent works by Baricz) concerns the 
subsequent derivatives of certain entire function $F$ of special type. It is
   \[ \Delta_L^n (t) = [F^{(n)}(t)]^2 - F^{(n-1)}(t)\,F^{(n+1)}(t), \qquad n \in \mathbb N\,.\]
The estimate $\Delta_L^n (t) \geq 0, \, t\in I \subseteq \mathbb R$ we call Laguerre inequality of order $n$.  

In turn, taking $t \mapsto t^{\frac1\alpha}$ in the Prabhakar type model $X \sim {\rm ML}(\alpha, \beta, \gamma)$ we get a 
transformed  random variable $X_1$, say. Since the characteristic function (CHF) of $X_1$ reads
   \[ \varphi_t(x) = \mathbb E {\rm e}^{{\rm i}x X_1} = \frac{E_{\alpha, \beta}^\gamma\left(t{\rm e}^{{\rm i}x}\right)}
	                   {E_{\alpha, \beta}^\gamma(t)}\,, \qquad (t, x) \in \mathbb R_+ \times \mathbb R \,,\]
which implies
   \begin{align} \label{O1}
	    \mathbb EX_1 &= -{\rm i}\, \frac{\partial \varphi_t(x)}{\partial x}\Big|_{x=0}  
	                  = -\frac{\rm i}{E_{\alpha, \beta}^\gamma(t)}\, \frac{\partial}{\partial x}\, E_{\alpha, \beta}^\gamma
											\left(t{\rm e}^{{\rm i}x}\right)\Big|_{x=0} \notag \\
									 &=  \gamma t \, \frac{E_{\alpha, \alpha+\beta}^{\gamma+1}(t)}{E_{\alpha, \beta}^\gamma(t)}
									  =  \frac{t}{E_{\alpha, \beta}^\gamma(t)}\,\left(E_{\alpha, \beta}^\gamma(t)\right)'\,,
	 \end{align} 
from which we infer
   \begin{equation} \label{O2}
	    \frac{\partial}{\partial x}\, E_{\alpha, \beta}^\gamma\left(t{\rm e}^{{\rm i}x}\right)\Big|_{x=0} 
			              = {\rm i}\,t\,\left(E_{\alpha, \beta}^\gamma(t)\right)'\,.
	 \end{equation}
Similarly we have 
   \begin{align} \label{O3}
			\mathbb EX_1^2 &= -\frac{\partial^2 \varphi_t(x)}{\partial x^2}\Big|_{x=0} 
			                = - \frac1{E_{\alpha, \beta}^\gamma(t)}\,
										    \frac{\partial^2}{\partial x^2}E_{\alpha, \beta}^\gamma\left(t{\rm e}^{{\rm i}x}\right) \Big|_{x=0} \notag \\
			               &= \frac1{E_{\alpha, \beta}^\gamma(t)} \left\{ \gamma t\, E_{\alpha, \alpha+\beta}^{\gamma+1}(t) 
										  + \gamma(\gamma+1)t^2\, E_{\alpha, 2\alpha+\beta}^{\gamma+2}(t)\right\}\,,
	 \end{align}
and since 
   \begin{equation} \label{X1}
	    \left(E_{\alpha, \beta}^\gamma(t)\right)'' = \gamma(\gamma+1) \, E_{\alpha, 2\alpha+\beta}^{\gamma+2}(t)\,,
	 \end{equation}
we conclude
   \begin{equation} \label{O4}
	    \frac{\partial^2}{\partial x^2}E_{\alpha, \beta}^\gamma\left(t{\rm e}^{{\rm i}x}\right) \Big|_{x=0} 
			     + t\,\left(E_{\alpha, \beta}^\gamma(t)\right)' + t^2\,\left(E_{\alpha, \beta}^\gamma(t)\right)'' = 0 \, .
	 \end{equation}
Now, we are ready to formulate our next set of results.

\begin{theorem} For all $\min\{\alpha, \beta, \gamma\}>0$ and $t>0$ there holds the first order Laguerre inequality:
   \begin{equation} \label{O5} 
      \Delta_L(t) \equiv \left[\frac{\partial}{\partial x}E_{\alpha, \beta}^\gamma\left(t{\rm e}^{{\rm i}x}\right) \Big|_{x=0}\right]^2 
			            - E_{\alpha, \beta}^\gamma(t) \cdot \frac{\partial^2}{\partial x^2}E_{\alpha, \beta}^\gamma
									  \left(t{\rm e}^{{\rm i}x}\right) \Big|_{x=0} \geq 0,
	 \end{equation}
Moreover, for the Laguerre difference the following functional upper bounds hold 
   \begin{equation} \label{O6}
	    \Delta_L(t) \leq \begin{cases}
			                    \dfrac{\gamma}{ \Gamma_0\Gamma_1\, t(1-t)^{2\gamma+1}} \hspace{2.9cm} 
													\min\{\alpha, \beta, \gamma\}>0,\, t\in (0,1) \\ \\
													\dfrac{\gamma\,{}_1F_1(\gamma; \beta; t)\,{}_1F_1(\gamma+1; \alpha+\beta; t)}
													{t\,\Gamma(\beta)\, \Gamma(\alpha+\beta)}	\qquad \alpha \geq 1,\, \beta \geq t_0;\, t>0
											 \end{cases}\,,
	 \end{equation}
where $\Gamma_0, \Gamma_1, t_0$ remain unchanged from  {\rm Theorem 3}.
\end{theorem}
	
\begin{proof} By the moment inequality $\mathbb EX_1^2 \geq \big[\mathbb EX_1\big]^2$, \eqref{O1} and \eqref{O3} we deduce the 
first assertion of the theorem. Next, replacing the partial derivatives in \eqref{O5} with the derivatives in $t$ from \eqref{O2} 
and \eqref{O4} we immediately get 
   \begin{equation} \label{X2}
	    E_{\alpha, \beta}^\gamma(t) \left[ \left(E_{\alpha, \beta}^\gamma(t)\right)'
	            + t\,\left(E_{\alpha, \beta}^\gamma(t)\right)''\right] \geq t\,\left[ \left(E_{\alpha, \beta}^\gamma(t)\right)' \right]^2\,, 
	 \end{equation}
or equivalently
   \[ t\,\Delta_L(t) \leq E_{\alpha, \beta}^\gamma(t) \cdot \left(E_{\alpha, \beta}^\gamma(t)\right)'
	                = \gamma \, E_{\alpha, \beta}^\gamma(t) \cdot E_{\alpha, \alpha+\beta}^{\gamma+1}(t)\, .\]
Now, applying the estimates \eqref{N10}, \eqref{N11} (having in mind to use $t \mapsto t^{\frac1\alpha}$) to the 
right--hand--side expression respectively, we arrive at the second statement \eqref{O6} of the theorem.
\end{proof} 

\begin{remark} 
We point out that \eqref{N9} is in fact equivalent to \eqref{X2}. \hfill $\Box$
\end{remark}

Finally it is worth to mention two recurrence properties of the Prabhakar function's $t$--derivatives. 

\begin{proposition} For all $\min\{\alpha, \beta, \gamma\}>0$ and $t>0$ there hold:
   \begin{align*} 
	    \alpha\gamma \,t\, E_{\alpha, \alpha+\beta+1}^{\gamma+1}(t) &= E_{\alpha, \beta}^\gamma(t) - \beta\, E_{\alpha, \beta+1}^\gamma(t)\\ 
			\alpha^2 \gamma (\gamma+1)\,t^2\, E_{\alpha, 2\alpha+\beta+2}^{\gamma+2}(t) &= E_{\alpha, \beta}^\gamma(t) 
			               - (\alpha+2\beta+1)\, E_{\alpha, \beta+1}^\gamma(t)\\ 
										&\qquad + (\alpha+\beta+1)(\beta+1)\, E_{\alpha, \beta+2}^\gamma(t)\, . 
	 \end{align*}
\end{proposition}

\begin{proof} Indeed, by direct calculation we show the identity
   \[ \alpha \,t\, \left(E_{\alpha, \beta+1}^\gamma(t)\right)' = E_{\alpha, \beta}^\gamma(t) - \beta\, E_{\alpha, \beta+1}^\gamma(t).\]
Since $\big(E_{\alpha, \beta+1}^\gamma(t)\big)' = \gamma E_{\alpha, \alpha+\beta+1}^{\gamma+1}(t)$ (see \eqref{O1}), we have
   \[ \alpha \gamma\,t\, E_{\alpha, \alpha+\beta+1}^{\gamma+1}(t) = E_{\alpha, \beta}^\gamma(t) - \beta\, E_{\alpha, \beta+1}^\gamma(t).\] 
Differentiating this equation, previously adapted with $\beta \mapsto \beta+1$ and multiplying with $\alpha t>0$, taking into 
account \eqref{X1}, we get 
	\begin{align*}
		  \alpha^2 \gamma\, t^2\,\left(E_{\alpha, \alpha+\beta+2}^{\gamma+1}(t)\right)' 
	        &= \alpha^2 \gamma(\gamma+1)\,t^2\, E_{\alpha, 2\alpha+\beta+2}^{\gamma+2}(t) \\
	        &= E_{\alpha, \beta}^\gamma(t) - (\alpha+2\beta+1)\, E_{\alpha, \beta+1}^\gamma(t) \\
					&\qquad + (\alpha+\beta+1)(\beta+1)\, E_{\alpha, \beta+2}^\gamma(t) \,, 
	 \end{align*}
which finishes the proof of the second recurrence.
\end{proof}

\section*{Acknowledgement}
\v Zivorad Tomovski acknowledges NWO grant 2015/2016, Department of Applied Mathematics, TU Delft.

\end{document}